\documentclass[reqno]{amsart}
\usepackage{amssymb, latexsym}

\usepackage{xspace}
\usepackage[bookmarksnumbered,colorlinks]{hyperref}
\usepackage{graphics}
\def\bb#1\eb{\textcolor{blue}
{#1}} %
\def\br#1\er{\textcolor{red}
{#1}} %
\def\bv#1\ev{\textcolor{green}
{#1}} %
\def\bc#1\ec{\textcolor{cyan}
{#1}} %

\usepackage{graphics}

\def\bal#1\eal{\begin{align}#1\end{align}}                      %
\def\baln#1\ealn{\begin{align*}#1\end{align*}}          
\def\bml#1\eml{\begin{multline}#1\end{multline}}        %
\def\bmln#1\emln{\begin{multline*}#1\end{multline*}}  %
\def\bga#1\ega{\begin{gather}#1\end{gather}}
\def\bgan#1\egan{\begin{gather*}#1\end{gather*}}

\makeatletter

\@addtoreset{equation}{section} \makeatother
\newtheorem{theorem}{Theorem}[section]
\newtheorem{lemma}[theorem]{Lemma}
\newtheorem{proposition}[theorem]{Proposition}
\newtheorem{corollary}[theorem]{Corollary}

\newtheorem{definition}[theorem]{Definition}
\newtheorem{example}[theorem]{Example}
\newtheorem{remark}[theorem]{Remark}

\newcommand{\m}{{\mathcal M}}
\newcommand{\J}{{\mathcal J}}
\newcommand{\s}{{\mathcal S}}
\renewcommand{\L}{{\mathbb L}}

\newcommand{\R}{{\mathbb R}}
\newcommand{\acca}{{\mathcal H}}
\newcommand{\elle}{{\mathcal L}}
\newcommand{\<}{\langle}
\renewcommand{\>}{\rangle}

\title{Remarks on global geodesic properties of G\"odel type spacetimes}
\thanks{This paper acknowledges the partial support of the Spanish
Grants with FEDER funds  MTM2010-18099 (MICINN). Furthermore, 
R. Bartolo and A.M. Candela acknowledge also the partial support of M.I.U.R. Research Project 
PRIN2009 ``Metodi Variazionali  e Topologici nello Studio di Fenomeni Nonlineari''
and of the G.N.A.M.P.A. Research Project 2011 ``Analisi Geometrica sulle Variet\`a
di Lorentz ed Applicazioni alla Relativit\`a Generale''; J.L. Flores acknowledges also the partial support of the Regional J.
Andaluc\'{\i}a Grant P09-FQM-4496, with FEDER funds.}

\author[R. Bartolo]{R. Bartolo}
\address{Rossella Bartolo\hfill\break\indent
Dipartimento di Ingegneria Meccanica e Gestionale\hfill\break\indent
Politecnico di Bari\hfill\break\indent
Via E. Orabona 4, 70125 Bari\hfill\break\indent
Italy}
\email{r.bartolo@poliba.it}

\author[A.M. Candela]{A.M. Candela}
\address{Anna Maria Candela \hfill\break\indent
Dipartimento di Matematica\hfill\break\indent
Universit\`a degli Studi di Bari ``Aldo Moro''\hfill\break\indent 
Via  Orabona 4, 70125 Bari\hfill\break\indent
Italy}
\email{candela@dm.uniba.it}

\author[J.L.
Flores]{J.L. Flores}
\address{Jos\'e Luis Flores \hfill\break\indent
Departamento de \'Algebra, Geometr\'{\i}a y
Topolog\'{\i}a, Facultad de Ciencias
\hfill\break\indent Universidad de M\'alaga \hfill\break\indent
Campus Teatinos, 29071 M\'alaga\hfill\break\indent
Spain}
\email{floresj@agt.cie.uma.es}



\date{}

\begin{document}

\begin{abstract}
The aim of this paper is to review and
complete the study of geodesics on G\"odel type spacetimes
initiated in \cite{CS} and improved in \cite{BCF}. In particular,
we prove some new results on geodesic connectedness and geodesic
completeness for these spacetimes.

\end{abstract}

\maketitle

\section{Introduction}
Classical critical points theorems and standard Morse theory are
directly applicable to functionals bounded from below which
satisfy compactness assumptions, such as the Palais--Smale
condition (see Section \ref{tools}), and whose critical points
have finite Morse index. Unluckily, these tools cannot be applied
to many interesting problems involving functionals that are
strongly indefinite. For example, geodesics joining two points
$z_p$, $z_q$ on an indefinite semi--Riemannian manifold $(\m,
\langle\cdot,\cdot\rangle_L)$ are the critical points of the
strongly indefinite $C^1$ action functional
\begin{equation}\label{act}
f(z)= \int_0^1\langle\dot z,\dot z\rangle_L{\rm d}s
\end{equation}
defined on the Hilbert manifold $\Omega$ of all the $H^1$--curves
joining $z_p$ to $z_q$ in $\m$ (for more details, see Section
\ref{tools}). Anyway, starting from the seminal paper \cite{BF},
in some particular settings, and according to the properties of
the manifold $\m$ and its indefinite metric
$\langle\cdot,\cdot\rangle_L$, the functional $f$ in (\ref{act})
has been widely studied by using variational methods, also
obtaining sometimes optimal results at least in the Lorentzian
case (we refer to the updated survey paper \cite{CS2} and
references therein). A typical situation occurs when the
Lorentzian metric tensor $\langle\cdot,\cdot\rangle_L$ presents
symmetries (i.e., Killing vector fields): one gets rid of the
negative contributions in the directions of the Killing fields
and, by means of some variational principles, it is possible to
handle with simpler functionals, which essentially depend only on
a Riemannian metric, so that they are bounded from below and
satisfy the Palais--Smale condition under reasonable assumptions.
This is the case of standard stationary and G\"odel type
spacetimes.

\begin{definition}\label{type1}
{\rm A Lorentzian manifold $(\m, \langle\cdot,\cdot\rangle_L)$ is
a {\em standard stationary spacetime} if there exist
a smooth finite--dimensional Riemannian manifold $(\m_0,\langle\cdot,\cdot\rangle_R)$,
a vector field $\delta$ and a positive smooth function $\beta$ on $\m_0$
such that $\m=\m_0\times\R$ and the Lorentz metric (under natural identifications)
is
\begin{equation}\label{metrica}
\langle\cdot,\cdot\rangle_L\ =\
\langle\cdot,\cdot\rangle_R + 2\langle\delta(x),\cdot\rangle_R \ dt - \beta(x)\ dt^2.
\end{equation}
When the cross term vanishes ($\delta\equiv 0$) the spacetime is called {\em standard static}.
This is a warped product $\m_0\times_{\sqrt\beta}\R$ with Riemannian base and negative definite fiber.}
\end{definition}
Recall that every stationary spacetime (i.e., a spacetime
admitting a timelike Killing vector field $K$) is locally a
standard stationary one with $K = \partial_t$.

On the other hand, G\"odel type spacetimes are Lorentzian
manifolds admitting a pair of commuting Killing vector fields
which span a two dimensional distribution where the metric has
index $1$ (the causal characters of the Killing vectors could
change on the manifold, see \cite[Example 5.1]{CSS}). More precisely,
we use the following definition (according to \cite{CS}):

\begin{definition}\label{type}
{\rm A Lorentzian manifold $(\m, \langle\cdot,\cdot\rangle_L)$ is
a {\em G\"odel type spacetime}, briefly $(GTS)$, if a
smooth finite--dimensional Riemannian manifold
$(\m_0,\langle\cdot,\cdot\rangle_R)$ exists such that $\m = \m_0 \times
\R^2$ and the metric $\langle\cdot,\cdot\rangle_L$ is described
as
\begin{equation}
\label{metric} \langle\cdot,\cdot\rangle_L\ =\
\langle\cdot,\cdot\rangle_R + A(x) dy^2 + 2 B(x) dy dt  - C(x)dt^2 ,
\end{equation}
where $x \in \m_0$, the variables $(y,t)$ are the natural
coordinates of $\R^2$ and $A$, $B$, $C$ are $C^1$ scalar fields on
$\m_0$ satisfying
\begin{equation}
\label{stima1} H(x)\ =\ B^2(x) + A(x) C(x) > 0\quad \mbox{for
all $x \in \m_0$.}
\end{equation}}
\end{definition}

Let us observe that condition (\ref{stima1}) implies that metric
(\ref{metric}) is Lorentzian.
It is also interesting to point out that $(GTS)$ are not necessarily
time--orientable (e.g., cf. \cite[Remark 1.2]{CSS}).

In \cite{g} G\"odel gives an exact solution of Einstein's field
equations with homogeneous perfect fluid distribution, the
so-called classical G\"odel Universe. This spacetime, described in
Example \ref{esem}$(e_1)$ below (see also \cite{cw, k} where its
geodesic equations are explicitly integrated), has a five
dimensional group of isometries, is geodesically complete and
admits closed causal curves (e.g., cf. \cite{he}). In \cite{rt}
Raychaudhuri and Thakurta start the study of homogeneity
properties of $(GTS)$ investigating homogeneity conditions of a
class of cylindrically symmetric metrics; later on, in \cite{ret}
Rebou\c{c}as and Tiomno introduce a definition for G\"odel metrics
in four dimensions and study their homogeneity conditions (see
also \cite{cdt, rtt}).

\begin{example}\label{esem}
{\rm The class of $(GTS)$ depicted in Definition \ref{type} is wide;
indeed this definition covers very different kinds of spacetimes,
including some physically relevant examples.
\begin{itemize}
\item[$(e_1)$] $\;\;$ The G\"odel Universe (cf. \cite{g}) is a
$(GTS)$ with
\[
\m_{0}=\R^{2},\quad \langle\cdot,\cdot\rangle_R = d x_1^2 + d x_2^2
\]
and with coefficients in (\ref{metric}) given by
\[
A(x) = - {\rm e}^{2 \sqrt 2 \omega x_1}/2, \quad B(x) = - {\rm e}^{\sqrt 2 \omega
x_1}, \quad C(x) \equiv 1
\]
($\omega >0$ represents the magnitude of the vorticity of the
flow). In \cite{CS}, by a direct integration of the geodesic
equations, it is constructed a geodesic joining each couple of
points in $\m$. 
\item[$(e_2)$]  $\;\;$ The G\"odel--Synge
spacetimes (cf. \cite{sant}) are $(GTS)$ with $\m_0=\R^{2}$ and
\[
\langle\cdot,\cdot\rangle_L = d x_1^2 + d x_2^2 - g(x_1) dy^2 - 2h(x_1) dy dt - dt^2,
\]
where $g, h$ are smooth functions of $x_1$ with $g>0$. If $2g=h^2$
and $h={\rm e}^{x_1}$, this metric reduces to the G\"odel
classical one. 
\item[$(e_3)$]  $\;\;$ Some Kerr--Schild spacetimes
(e.g., cf. \cite{ks}) are $(GTS)$ with again $\m_0 = \R^2$ and 
\[
\langle\cdot,\cdot\rangle_L = d x_1^2 + d x_2^2 +
dy^2 - dt^2 + V(x_1,x_2)(dy + dt)^2,
\]
where $V$ is an arbitrary function on $\R^2$. In this case, the
coefficients in (\ref{metric}) are
\[
A(x)= 1 + V(x),\quad B(x)=V(x), \quad C(x)= 1- V(x),
\]
and thus, $H(x)\equiv 1$. 
\item[$(e_4)$]  $\;\;$ Some standard
stationary spacetimes are $(GTS)$ with $\m =\m_0\times \R^2$,
being $(\m_0\times\R,\langle\cdot,\cdot\rangle_R + d y^2)$ the Riemannian part
and  
\[
\langle\cdot,\cdot\rangle_L = \langle\cdot,\cdot\rangle_R + d y^2 + 2 \delta(x) dy dt - \beta(x)dt^2
\]
the stationary metric with $\delta(x,y)\equiv \delta(x)\in\R$
and $\beta(x,y)\equiv\beta(x) > 0$ in $\m_0\times\R$.
Clearly, they are $(GTS)$ with $A(x) \equiv 1$, $B(x) = \delta(x)$ and $C(x) = \beta(x)$.\\
Vice versa, some $(GTS)$ are standard
stationary spacetimes with the product $A(x) C(x)>0$ on $\m_0$,
being standard static if, in addition, $B\equiv 0$. 
For example, if $A(x)>0$ on $\m_0$, the spatial part of the stationary
spacetime corresponds to $\m_0\times\R$ equipped with the
Riemannian metric $\langle\cdot,\cdot\rangle_R + A(x)dy^2$ (which
is complete if so is $\langle\cdot,\cdot\rangle_R$), the vector
field becomes $\delta(x,y)=(0,B(x))\in T\m_0\times\R$ and the
scalar field is $\beta(x,y) = C(x)>0$ for each $(x,y)\in
\m_0\times\R$. 
\item[$(e_5)$]  $\;\;$ Some examples of general
plane fronted waves are also $(GTS)$. More precisely, a {\em general
plane fronted wave} is a Lorentzian manifold $\m_0\times\R^2$
equipped with the metric
\[
\langle\cdot,\cdot\rangle_L=\langle\cdot,\cdot\rangle_R  + 2 dy dt + H_0(x,t) dt^2,
\]
where $(\m_0,\langle\cdot,\cdot\rangle_R)$ is a Riemannian
manifold, $(y,t)$ are the natural coordinates of $\R^2$ and the
smooth scalar field $H_0$ on $\m_0\times\R$ satisfies $H_0\not\equiv
0$. Clearly, when $H_0(x,t)$ is autonomous (i.e., it does not depend
on $t$), this spacetime is a $(GTS)$.\\
Results on geodesic completeness and connectedness for these
spacetimes can be found in \cite{CFS2}.
\end{itemize}}
\end{example}

The importance of the spacetimes above justifies the study of global
properties such as geodesic connectedness and geodesic completeness.
However, one cannot expect to prove
general results: indeed the global properties depend on the behaviour of the coefficients
(see respectively Theorems \ref{t1} and \ref{teocom} and related comments). 
Also the study of global hyperbolicity of $(GTS)$ is interesting but,
except in the case of classical G\"odel Universe, this is still an open problem
even if we are lead to think that at least the
chronology condition should be verified. This idea comes out as in
most of the previous references, when global properties can be
studied by variational methods, then the stated assumptions imply
also global hyperbolicity (cf. \cite{CS2} and references therein).
In fact, for example in the particular case of standard stationary 
spacetimes, we know sufficient conditions,
involving the coefficients of the metric, which ensure global
hyperbolicity (cf. \cite[Corollary 3.4]{san}), while recently a
characterization has been proven (see \cite[Theorem
4.3]{cjs}). 

The paper is organized as follows. In Section \ref{tools} we
recall some variational principles for geodesics on static
spacetimes and $(GTS)$. In Section \ref{geo} we present a new result
on geodesic connectedness, and compare it with the previous ones
in \cite{BCF}, showing its accuracy by examples. In Section
\ref{sec:1} we deal with geodesic completeness, and state a
sufficient condition in order to obtain it. Finally, in the
Appendix we fix some widely known notations about the variational
set up.

\section{The variational principle}\label{tools}

According to notations and statements contained in the Appendix,
there is a correspondence between geodesics joining two given
points $z_p$, $z_q$ on a semi--Riemannian manifold $(\m,
\langle\cdot,\cdot\rangle_L)$ and critical points of the action
functional $f$ in (\ref{act}) on the Hilbert manifold
$\Omega^1(z_p,z_q)$. As already remarked, if
$\langle\cdot,\cdot\rangle_L$ is not Riemannian then $f$ is
strongly indefinite, but, in some Lorentzian manifolds, this
difficulty can be overcome by introducing a new suitable
functional.

The kernel of our approach is a variational principle stated in
\cite[Theorem 2.1]{BFG} for static Lorentzian manifolds
$\m=\m_0\times\R$, with $\langle\cdot,\cdot\rangle_L$ as in
(\ref{metrica}) and $\delta\equiv 0$. It is based on the fact that
$\langle \partial_t,\dot z\rangle_L$ is constant along each
geodesic $z$, because of the Killing character of $\partial_t$.
Namely, $z_{p}=(x_p,t_p)$, $z_{q}=(x_q,t_q) \in \m$ are connected
by a geodesic $\bar z=(\bar x,\bar t)$, which is a critical point
of functional $f$ in (\ref{act}) on
$\Omega^{1}(z_{p},z_{q})=\Omega^1(x_p,x_q)\times W(t_p,t_q)$, if
and only if $\bar x$ is a critical point of the functional
\begin{equation}\label{newfunc}
J(x)\ =\ \frac 12\int_0^1 \langle \dot x,\dot x\rangle_R \, {\rm
d}s-\frac{\Delta_t^2}{2}\left(\int_0^1 \frac 1{\beta(x)}\, {\rm
d}s\right)^{-1}\quad\hbox{on
$\Omega^1(x_p,x_q)$,}
\end{equation}
with $\Delta_t:=t_p-t_q$ (for the stationary case we refer to \cite[Section 4.2]{CS2} and
references therein).

Next, let us consider the more general setting of $(GTS)$ with
$\m=\m_0\times\R^2$ and $\langle\cdot,\cdot\rangle_L$ as in
Definition \ref{type}. For each $x\in H^1(I,\m_0)$ let us define
\begin{eqnarray}
& &a(x) = \int_0^1 \frac{A(x)}{H(x)} \; {\rm d}s ,\;\; b(x) =
\int_0^1 \frac{B(x)}{H(x)}\; {\rm d}s , \;\; c(x) = \int_0^1
\frac {C(x)}{H(x)}\; {\rm d}s, \label{cdotrho}
\\
& &\elle(x) = b^2(x) + a(x) c(x). \label{cdotrho1}
\end{eqnarray}
As every $(GTS)$ admits two Killing vector fields $\partial_y$,
$\partial_t$ (not necessarily timelike), an extension of the
previous variational principle can be stated (cf.
\cite[Proposition 2.2]{CS}). In this setting, fixing $z_p =
(x_p,y_p,t_p)$, $z_q = (x_q,y_q,t_q) \in \m$, with $x_p$, $x_q \in
\m_0$ and $(y_p,t_p)$, $(y_q,t_q) \in \R^2$, we have that $\bar z
:I\rightarrow \m$ is a geodesic joining $z_p$ to $z_q$ in $\m$ if
and only if it is a critical point of the action functional
(\ref{act}), with $\langle\cdot,\cdot \rangle_L$ as in
(\ref{metric}), defined on the manifold $\Omega^1(z_p,z_q)=
\Omega^1(x_p,x_q) \times W(y_p,y_q) \times W(t_p,t_q)$.\\
Let $x \in \Omega^1(x_p,x_q)$ be such that $\elle(x) \ne 0$ (cf.
(\ref{cdotrho1})). For all $s\in I$ we define
\begin{eqnarray*}
&&\begin{aligned} \phi_y(x)(s)\ :=\ & y_p + \ \frac{\Delta_y\ b(x)
- \Delta_t\ c(x)} {\elle(x)}\
\int_0^s \frac {B(x)} {\acca(x)}\; {\rm d}\sigma\\
&+\ \frac{\Delta_y\ a(x) + \Delta_t\ b(x)} {\elle(x)}\
\int_0^s \frac{C(x)} {\acca(x)}\; {\rm d}\sigma ,
\end{aligned}\label{uai}\\
&&\begin{aligned} \phi_t(x)(s)\ :=\ & t_p - \ \frac{\Delta_y\ b(x)
- \Delta_t\ c(x)} {\elle(x)}\
\int_0^s \frac{A(x)} {\acca(x)}\; {\rm d}\sigma\\
&+\ \frac{\Delta_y\ a(x) + \Delta_t\ b(x)} {\elle(x)}\ \int_0^s
\frac{B(x)}{\acca(x)}\; {\rm d}\sigma
\end{aligned}
\end{eqnarray*}
with $\Delta_y := y_q-y_p$ and $\Delta_t := t_q-t_p$. Standard
arguments imply that the functions $\phi_y$ and $\phi_t$, which go
from $\Omega^1(x_p,x_q)$ to $W(y_p,y_q)$ and $W(t_p,t_q)$, resp.,
are $C^{1}$.

Then, the following proposition holds (see
\cite[Proposition 2.2]{CS}).

\begin{proposition}\label{principe} Let $(\m,\langle\cdot,\cdot\rangle_L)$
be a $(GTS)$ and $x_p$, $x_q \in \m_0$ be such that $|\elle(x)| > 0$ for all $x
\in \Omega^1(x_p,x_q)$. Then, the following statements are equivalent:
\begin{itemize}
\item[{\sl (i)}] $\bar z \in \Omega^1(z_p,z_q)$ is a critical
point of the action functional $f$ in (\ref{act}); \item[{\sl
(ii)}] $\;$ setting $\bar z = (\bar x,\bar y,\bar t)$, the curve
$\bar x \in \Omega^1(x_p,x_q)$ is a critical point of the $C^1$
functional
\begin{equation}\label{newfunct1}
\J(x) = \frac 12\ \int_0^1 \langle\dot x,\dot x\rangle_R\; {\rm
d}s\  +\ \frac{\Delta_y^2 a(x) + 2 \Delta_y \Delta_t b(x) -
\Delta_t^2 c(x)} {2 \elle(x)}\quad\hbox{on $\Omega^1(x_p,x_q)$}
\end{equation}
(see (\ref{cdotrho})--(\ref{cdotrho1})), while $\bar y =
\phi_y(\bar x)$, $\bar t = \phi_t(\bar x)$, with $\phi_y$,
$\phi_t$ as above.
\end{itemize}
Furthermore,
\[
\J(x)\ =\ f(x,\phi_y(x),\phi_t(x)) \quad \hbox{for all $x \in
\Omega^1(x_p,x_q)$.}
\]
\end{proposition}
Thus, the geodesic connectedness problem in the standard static
and $(GTS)$ cases reduces to give conditions on the functionals $J$
in (\ref{newfunc}) and $\J$ in (\ref{newfunct1}), resp., which
allows us to apply the classical critical point theorem below
(see \cite[Theorem 2.7]{rab}).

\begin{theorem}\label{min}
Assume that $\Omega$ is a complete Riemannian manifold
and $F$ is a $C^1$ functional on $\Omega$ which
satisfies the Palais--Smale condition, i.e. any sequence $(x_k)_k
\subset \Omega$ such that
\[
(F(x_k))_k\; \mbox{is bounded}\quad\hbox{and}\quad \lim_{k \to
+\infty}F'(x_k) = 0,
\]
converges in $\Omega$, up to subsequences. Then, if $F$
is bounded from below, it attains its infimum.
\end{theorem}

\begin{remark}
{\rm In order to obtain a multiplicity result on geodesics joining
two fixed points in standard static spacetimes or $(GTS)$, the
Ljusternik--Schnirelman theory can be applied to $J$ in
(\ref{newfunc}) or $\J$ in (\ref{newfunct1}) whenever the
Riemannian part has a ``rich topology'' (for the static case see
\cite{bcfs} and references therein, for $(GTS)$ see \cite{BCF, CS,
CS1}).}
\end{remark}

In order to avoid technicalities, hereafter we assume that $\m_0$
is complete, so that $\Omega^{1}(x_p,x_q)$ is also complete for
each $x_p,x_q\in\m_0$. Then, let us recall the following result (
cf. \cite[Proposition 4.3]{bcfs} and \cite[Lemma 5.3]{BCF}).

\begin{lemma}\label{ps1}
Let $x_p$, $x_q$ be two points of a given complete Riemannian manifold
$(\m_{0},\langle\cdot,\cdot \rangle_R)$. 
Then, the following
statements hold:
\begin{itemize}
\item[$(a)$]  if $\m= \m_0\times \R$ is a static Lorentzian manifold and
$J$ in (\ref{newfunc}) is coercive on $\Omega^1(x_p,x_q)$, then
$J$ satisfies the Palais--Smale condition on $\Omega^1(x_p,x_q)$;
\item[$(b)$] if $\m= \m_0\times \R^2$ is a $(GTS)$,
 $\J$ in (\ref{newfunct1}) is coercive on $\Omega^1(x_p,x_q)$, and
there exists $\nu > 0$ such that
\[
|\elle(x)| \ \ge \nu\quad \hbox{for all $x \in
\Omega^1(x_p,x_q)$,}
\]
then $\J$ satisfies the Palais--Smale condition on
$\Omega^1(x_p,x_q)$.
\end{itemize}

\end{lemma}

Summing up, geodesic connectedness of the mentioned spacetimes is
guaranteed by conditions implying the coercivity and lower
boundedness of the ``Riemannian'' functional associated to the problem.

For instance, in the case of $J$ in (\ref{newfunc}), these
conditions correspond to restrictions on the growth of the
(positive) metric coefficient $\beta$ in (\ref{metrica}): $\beta$
bounded in the pioneer paper \cite{BFG} or, more in general, 
$\beta$ subquadratic or growing at most
quadratically with respect to the distance $d(\cdot,\cdot)$
induced on $\m_0$ by its Riemannian metric
$\langle\cdot,\cdot\rangle_{R}$, i.e. existence of $\lambda \ge
0$, $k \in \R$ and a point $x_0\in \m_0$ such that
\begin{equation}\label{crescita}
0 < \beta(x) \le \lambda d^2(x, x_0) + k \quad \mbox{for all $x \in \m_0$,}
\end{equation}
(cf. \cite[Theorem 1.1]{bcfs} and references therein). Remarkably,
this second growth condition on $\beta$ is optimal, as showed in
\cite[Section 7]{bcfs} by constructing a family of geodesically
disconnected static spacetimes with superquadratic, but
arbitrarily close to quadratic, coefficients $\beta$.

\section{Geodesic Connectedness in $(GTS)$}\label{geo}

At a first glance the problem in $(GTS)$ can be handled in the same
manner as in the static case. However, we cannot expect optimality
by applying this variational approach. In fact, the classical
G\"odel Universe cannot be studied by our tools, due to the lack
of the assumption $\elle(x)\not=0$ on $\Omega^1(x_p,x_q)$ for each
couple of points $x_p,x_q\in\R^2$ (cf. Example \ref{esem}$(e_1)$).
In this section we state and prove a new theorem on geodesic
connectedness for $(GTS)$ (in addition to the previous ones in
\cite{BCF,CS}), which, even if is not optimal, is accurate
in the sense described below (see Corollary \ref{c2} and Example
\ref{ex}).

\begin{theorem}\label{t1}
Let $(\m= \m_0 \times \R^2,\langle\cdot,\cdot\rangle_L)$ be a
G\"odel type spacetime such that:
\begin{itemize}
\item[$(h_1)$] $\;\;$
\ $(\m_{0},\langle\cdot,\cdot\rangle_{R})$ is a complete
Riemannian manifold;
\item[$(h_2)$] $\;\;$
\ there exists $\nu>0$ such that
${\mathcal L}(x)\geq\nu>0$ for all $x\in H^{1}(I,\m_{0})$;
\item[$(h_3)$] $\;\;$
\ $m(x)\geq h(x) > 0$ for all $x\in H^{1}(I,\m_{0})$,
with $m(x):={\rm max}\{a(x),-c(x)\}$ and
\[
h(x):=\displaystyle{\int_0^1\frac{{\rm d}s}{\lambda
d^{2}(x(s),x_0)+k}}\qquad \hbox{for some $\lambda \ge 0$, $k\in\R$
and $x_0\in \m_0$.}
\]
\end{itemize}
Then, $(\m,\langle\cdot,\cdot\rangle_L)$ is geodesically
connected.
\end{theorem}

\begin{proof}
Let us take any $z_{p}=(x_p,y_p,t_p)$, $z_q=(x_q,y_q,t_q)\in\m$, with
$x_p$, $x_q\in \m_{0}$ and $(y_p,t_p)$, $(y_q,t_q)\in \R^{2}$.
From hypothesis $(h_2)$ (in particular $\elle(x)\not=0$),
Proposition \ref{principe} can be applied, and so the existence
of geodesics joining $z_p$ to $z_q$ reduces to find some
critical points of $\J$ in (\ref{newfunct1}) on
$\Omega^1(x_p,x_q)$. Following the arguments developed in
\cite[Section 5]{BCF}, we have that $\J$ can be written as
follows:
\begin{equation}\label{nuovoj}
{\J}(x)=\frac{1}{2}\|\dot{x}\|^{2}-\frac{1}{2}\frac{\Delta_{+}^{2}(x)}{\lambda_{-}(x)}
-\frac{1}{2}\frac{\Delta_{-}^{2}(x)}{\lambda_{+}(x)},
\end{equation}
where
\[
\|\dot{x}\|^{2} = \int_0^1 \langle\dot x,\dot x\rangle_R\; {\rm d}s,\quad
\lambda_{\pm}(x)=\frac{a(x)-c(x)\pm\sqrt{(a(x)+c(x))^{2}+4b(x)^{2}}}{2}
\]
and $\Delta_\pm(x)$ are suitable real maps depending also on
$\Delta_y$, $\Delta_t$ (see (\ref{cdotrho}) and \cite[p.
784]{BCF}). Since $\elle(x)= -\lambda_-(x)\lambda_+(x)$,
necessarily $\lambda_{+}(x)>0>\lambda_{-}(x)$ for all $x\in
\Omega^{1}(x_p,x_q)$, and thus
\[
{\J}(x)\geq
\frac{1}{2}\|\dot{x}\|^{2}-\frac{1}{2}\frac{\Delta_{-}^{2}(x)}{\lambda_{+}(x)}.
\]
Note also that, by the definition of $m(x)$ in $(h_3)$, we get
\[
\lambda_{+}(x)\geq \frac{a(x)-c(x)+|a(x)+c(x)|}{2}= m(x)>0.
\]
Hence, $(h_3)$ implies
\[
{\J}(x)\geq
\frac{1}{2}\|\dot{x}\|^{2}-\frac{\Delta_{-}^{2}(x)}{2m(x)}\geq\frac{1}{2}\|\dot{x}\|^{2}-\frac{\Delta_{-}^{2}(x)}{2}\left(h(x)\right)^{-1} \quad
\hbox{for all $x\in \Omega^{1}(x_p,x_q)$.}
\]
So, from \cite[Theorem 1.1]{bcfs}, it follows that ${\J}$ is
bounded from below and coercive (cf. (\ref{newfunc}) and
(\ref{crescita})). Furthermore, by $(h_2)$ and Lemma
\ref{ps1}$(b)$, the functional ${\J}$ satisfies the
Palais--Smale condition. Hence, Theorem \ref{min} can be applied,
and a geodesic connecting $z_{p}$ with $z_{q}$ is obtained. As
$z_p$, $z_q$ are arbitrary, the thesis follows. 
\end{proof}

An immediate consequence of Theorem \ref{t1} is
the following result concerning some standard stationary
spacetimes (cf. Example \ref{esem}$(e_4)$).
\begin{corollary}\label{c2} Let $(\m= \m_0 \times
\R^2,\langle\cdot,\cdot\rangle_{L})$ be a standard stationary
spacetime with
$\langle\cdot,\cdot\rangle_{L}=\langle\cdot,\cdot\rangle_{R}+dy^{2}+2\delta(x)
dy\, dt-\beta (x) dt^{2}$, where $\delta$, $\beta :\m_0 \to\R$, $\beta(x) > 0$ in $\m_0$. 
Assume also that
\begin{itemize}
\item[$(s_1)$] $\;\;$
$(\m_{0},\langle\cdot,\cdot\rangle_{R})$ is a complete
Riemannian manifold;
\item[$(s_2)$] $\;\;$ there exist
$\lambda_{1},\lambda_{2}\geq 0$, $k_{1},k_{2}\in \R$ and a point
$x_0\in\m_{0}$ such that
\[
\beta(x)\leq \lambda_1 d^{2}(x,x_0)+k_1,\quad \delta(x)\leq
\lambda_2 d(x,x_0)+k_2\qquad \hbox{for all $x\in \m_{0}$}.
\]
\end{itemize}
Then, $(\m= \m_0 \times \R^2,\langle\cdot,\cdot\rangle_{L})$ is
geodesically connected.
\end{corollary}

\begin{proof}
As the standard stationary spacetime $(\m= \m_0 \times
\R^2,\langle\cdot,\cdot\rangle_{L})$ is a $(GTS)$ with $A(x)\equiv
1$, $B(x)=\delta (x)$ and $C(x)=\beta(x)$, the thesis follows from
Theorem \ref{t1}. 
\end{proof}

Notice that Corollary \ref{c2} is a particular case of
\cite[Theorem 1.2]{bcf1} for general standard stationary manifolds
$\m=\m_0\times\R$ with $\langle\cdot,\cdot\rangle_{L}$ as in
(\ref{metrica}), which proof is based on fine estimates involving
the metric coefficients. The following example shows the accurate
character of this result.
\begin{example}\label{ex}
{\rm 
Let us consider $\R^{3}$ endowed with the
following family of metrics:
\[
\langle\cdot,\cdot\rangle_{L, \varepsilon}=dx^{2}+dy^{2}-\beta_{\epsilon}(x)dt^{2},\quad
\epsilon\geq 0,
\]
where $(x,y,t)\in\R^3$ and  $\beta_{\epsilon}$ is a (positive)
smooth function on $\R$ such that
\[
\left\{\begin{array}{ll} \beta_{\epsilon}(x)=1+|x|^{2+\epsilon} &
\hbox{if $x\in \R\setminus (-1,1)$} \\
\beta_{\epsilon}([-1,1])\subset [1,2]. & \end{array}\right.
\]
By Corollary \ref{c2} (with $\delta\equiv 0$), the spacetime is
geodesically connected if $\epsilon=0$. However, the spacetime is
geodesically disconnected for any (and thus, for arbitrarily close
to zero) strictly positive $\epsilon$ (see \cite[Section
7]{bcfs}). }
\end{example}

In order to give a more precise idea of the known results on
geodesic connectedness in $(GTS)$ by applying variational tools, let
us review the corresponding results in \cite{BCF}. In
\cite[Theorem 4.3]{BCF}, by using the expression (\ref{nuovoj}) of
$\J$, the geodesic connectedness of $(GTS)$ is proven under
assumptions $(h_1)$ and $(h_2)$ in Theorem \ref{t1}, in addition
to the following one:
\begin{itemize}
\item[$(h_3')$] $\;\;$
$A(x)-C(x)>0$ for all $x \in \m_0$
and the (positive) map
$\displaystyle{
\frac{H(x)}{A(x)-C(x)}}$ is at most quadratic.
\end{itemize}
Indeed, these conditions imply that $\J$ is bounded from below and
coercive, which allows us to apply
Theorem \ref{min} in view of Lemma \ref{ps1}$(b)$.

As an immediate application of this result to Kerr--Schild
spacetime (Example \ref{esem}$(e_3)$), observe that here
$A(x)-C(x)= 2V(x)$, $H(x)\equiv 1$ and $\elle(x)\not=0$ on
$H^1(I,\m_0)$. Thus, the geodesic connectedness is ensured if $V$
is strictly positive and $(2V(x))^{-1}$ is at most quadratic.

On the other hand, in \cite[Theorem 4.4]{BCF} we consider the simpler case,
where $\elle(x) \le -\nu < 0$ for all $x \in H^{1}(I,\m_0)$ and
$A(x)-C(x)<0$ for all $x \in \m_0$.

Finally, notice that in \cite{BCF} the growth assumption 
involves only the metric coefficients, and not the integrals in
(\ref{cdotrho}). This contrasts with \cite{CS, CS1}, where, in
order to get the coercivity of $\J$, it is required that
\[
\left|\frac{a(x)} {\elle(x)}\right|,\quad
\left|\frac {b(x)} {\elle(x)}\right|,\quad
 \left|\frac {c(x)} {\elle(x)}\right|
 \quad \hbox{are uniformly bounded on $H^1(I,\m_0)$.}
 \]

\begin{remark} {\rm 
Regarding to the case $A\equiv C$ left over in \cite{BCF}, if $A$
(hence $C$) is always different from zero,
then we are in the stationary case (Example \ref{esem}$(e_4)$) with $\beta(x) = |A(x)|$.\\
In general, if $B\equiv 0$ and $H(x) = A(x)C(x) > 0$ with $A(x) >
0$ and $\beta(x) = C(x)$, then we have Example \ref{esem}$(e_4)$
in the static case. So, $\J(x)\geq J(x)$ on each
$\Omega^{1}(x_p,x_q)$ and the optimal result in \cite[Theorem
1.1]{bcfs} can be used. Let us point out that a direct use of
$(h_3')$ for the particular case
$A\equiv 1$ would give the desired result only for $\beta(x) <1$.\\
If $A\equiv C\equiv 0$ then $\elle(x)=b^2(x)$ and $(GTS)$ becomes
 the more general type of warped product spacetimes, with fiber the
 two dimensional Lorentz--Minkowski spacetime $\L^2$
 (see also \cite{CS,CSS} and references therein).
In this case we deal again with a functional as in
(\ref{newfunc}), and we get global geodesic connectedness for the
class of metrics $\langle\cdot,\cdot\rangle_{R} -
2\delta(x)\,dy\,dt$, where $\delta$
is a positive function with at most a quadratic growth (compare with \cite[Appendix B]{CS}). \\
Moreover, if $a\equiv c$ on $H^{1}(I,\m_0)$, then
\[
\J(x)\geq \frac 12 \|\dot x\|^2- \frac{\Delta_-^2}{|a(x)|}.
\]
Hence, if $A(x) > 0$ in $\m_0$, we obtain geodesic connectedness by assuming that
$H(x)/A(x)$ grows at most qua\-dra\-ti\-cal\-ly in $\m_0$
(cf. (\ref{newfunc}) and (\ref{crescita})).}
\end{remark}

\begin{remark}{\rm 
In \cite{pt} Piccione and Tausk generalize the Morse index theorem
to semi--Riemannian manifolds admitting a smooth distribution
spanned by commuting Killing vectors fields and containing a
maximal negative distribution for the given metric. In paticular,
they obtain Morse relations for standard stationary
semi--Riemannian manifolds and, when the nondegeneracity condition
$|\elle(x)|>0$ holds, for $(GTS)$ (cf. \cite[Theorems 4.6 and
4.8]{pt}). Clearly, Morse relations can be obtained also under our
assumptions.}
\end{remark}

\section{Geodesic Completeness}
\label{sec:1}

In this section we establish and prove a result on geodesic
completeness for $(GTS)$.
\begin{theorem}\label{teocom}
Let $(\m= \m_0 \times \R^2,\langle\cdot,\cdot\rangle_L)$ be a
G\"odel type spacetime such that:
\begin{itemize}
\item[$(c_1)$] $\;\;$ $(\m_0,\langle\cdot,\cdot\rangle_R)$ is a complete
Riemannian manifold;
\item[$(c_2)$] $\;\;$ there exist $\lambda\geq 0$, $k\in \R$ and a point
$x_0\in \m_0$ such that the (positive) map
\[
\mu: x\in\m_0\mapsto
C(x)-A(x)+\sqrt{(A(x)+C(x))^{2}+4B(x)^{2}}\in \R
\]
satisfies
\begin{equation}\label{e4}
1/\mu(x)\leq \lambda d^{2}(x,x_0)+k\quad\hbox{for all
$x\in \m_{0}$.}
\end{equation}
\end{itemize}
Then, $(\m,\langle\cdot,\cdot\rangle_L)$ is geodesically complete.
\end{theorem}

\begin{proof}
Let $z:[0,T)\rightarrow \m$, $z(s)=(x(s),y(s),t(s))$, be
an inextendible geodesic. Arguing by contradiction, it is enough
to prove that if $T < +\infty$ then the
$\langle\cdot,\cdot\rangle_{R}$--length of $x(s)$ is bounded, and
so, $z$ can be extended to $T$ against the maximality
assumption (see \cite[Lemma 5.8]{O}).\\
As $\partial_y$ and $\partial_t$ are Killing vector fields, there
exist constants $c_1, c_2\in \R$ such that
\begin{equation}\label{e1}
\left\{
\begin{array}{l}
A(x)\dot{y} + B(x)\dot{t}\equiv c_1\\
B(x)\dot{y} - C(x)\dot{t}\equiv c_2
\end{array}
\right.\qquad \hbox{for all $s \in [0,T)$,}
\end{equation}
with
\begin{equation}\label{e0}
\s(x)\ =\ \left( \begin{array}{cc}
A(x) & B(x) \\
B(x) & -C(x)
\end{array} \right)
\end{equation}
symmetric matrix with $\det \s(x)=-H(x)<0$.\\
Furthermore, as $z$ is a geodesic, there exists a constant $E_z\in \R$
such that
\begin{equation}\label{e2}
\<\dot{z},\dot{z}\>_{L}=\<\dot{x},\dot{x}\>_{R}+A(x)\dot{y}^{2}+
2B(x)\dot{y}\dot{t}-C(x)\dot{t}^{2}\equiv E_z\quad \hbox{for all $s\in [0,T)$.}
\end{equation}
Thus, by (\ref{e1}) and (\ref{e2}) we get
\begin{equation}\label{e02}
\<\dot{x},\dot{x}\>_{R}+c_1\dot{y}+c_2\dot{t}=E_z\quad \hbox{for all $s\in [0,T)$.}
\end{equation}
On the other hand, by (\ref{e1}) and (\ref{stima1}) we have
\[
\dot{y}=\frac{c_1 C(x)+c_2 B(x)}{H(x)},\qquad \dot{t}=\frac{c_1
B(x)-c_2 A(x)}{H(x)}.
\]
Whence, by (\ref{e02}) and using the notation
$\|\dot{x}\|^{2}_{R}:=\<\dot{x},\dot{x}\>_R$, we get
\begin{equation}\label{e3}
\|\dot{x}\|^{2}_{R} \ = \ E_z+\frac{c_2^{2}A(x)-c_1^{2}C(x)-2c_{1}c_{2}B(x)}{H(x)}.
\end{equation}
Note that the symmetric matrix $\s(x)$ in (\ref{e0}) admits two
(non--null) real eigenvalues
\[
\Lambda_\pm(x)=\frac{A(x)-C(x) \pm \sqrt{(A(x)+C(x))^2 +
4B^2(x)}}{2},\qquad\hbox{with}\; \Lambda_{+}(x)>0>\Lambda_{-}(x).
\]
Recall that by standard arguments there exists an orthogonal matrix $Q(x)$ such that
\[
Q(x)^{T}\left( \begin{array}{cc}
A(x) & B(x) \\
B(x) & -C(x)
\end{array} \right)Q(x)=\left( \begin{array}{cc}
\Lambda_{+}(x) & 0 \\
0 & \Lambda_{-}(x)
\end{array} \right).
\]
Let us denote
$(\tilde{c}_{1},\tilde{c}_{2})=(c_{1}^{2}+c_{2}^{2})^{-
{1/2}}(c_{1},c_{2})$ and
$(\tilde{c}_{1}(x),\tilde{c}_{2}(x))=(\tilde{c}_{1},\tilde{c}_{2})Q(x)$.
By definition we have $\mu(x) = -2 \Lambda_{-}(x)$, and, by the
orthogonality of $Q(x)$, we have $[\tilde{c}_i(x)]^2 \le 1$ for $i
\in\{1,2\}$. So, we can rewrite (\ref{e3}) as:
\[
\begin{array}{rl}
\|\dot{x}\|^{2}_{R}\ = &\ \displaystyle{E_z+\frac{\left( \begin{array}{cc} c_1
c_2
\end{array} \right)\left( \begin{array}{cc}
A(x) & B(x) \\
B(x) & -C(x)
\end{array} \right)\left( \begin{array}{cc}
c_1  \\
c_2
\end{array} \right)}{H(x)}} \\
=&\ \displaystyle{E_z+\frac{\left( \begin{array}{cc} c_1 & c_2
\end{array} \right)Q(x)\left( \begin{array}{cc}
\Lambda_{+}(x) & 0 \\
0 & \Lambda_{-}(x)
\end{array} \right)Q(x)^{T}\left( \begin{array}{cc}
c_1  \\
c_2
\end{array} \right)}{H(x)} }\\
=&\ \displaystyle{E_z+(c_{1}^{2}+c_{2}^{2})\frac{\left( \begin{array}{cc} \tilde{c}_1 & \tilde{c}_2
\end{array} \right)Q(x)\left( \begin{array}{cc}
\Lambda_{+}(x) & 0 \\
0 & \Lambda_{-}(x)
\end{array} \right)Q(x)^{T}\left( \begin{array}{cc}
\tilde{c}_1  \\
\tilde{c}_2
\end{array} \right)}{H(x)} }\\
=&\ \displaystyle{E_z+(c_{1}^{2}+c_{2}^{2})\frac{\left( \begin{array}{cc} \tilde{c}_1(x) &
\tilde{c}_2(x)
\end{array} \right)\left( \begin{array}{cc}
\Lambda_{+}(x) & 0 \\
0 & \Lambda_{-}(x)
\end{array} \right)\left( \begin{array}{cc}
\tilde{c}_1(x)  \\
\tilde{c}_2(x)
\end{array} \right)}{-\Lambda_{+}(x)\Lambda_{-}(x)} }\\
=&\ \displaystyle{E_z - (c_{1}^{2}+c_{2}^{2})\left(\frac{[\tilde{c}_1(x)]^2}{\Lambda_{-}(x)}
\ +\ \frac{[\tilde{c}_2(x)]^2}{\Lambda_{+}(x)}\right)}\\
\le&\ \displaystyle{E_z - \frac{c_{1}^{2}+c_{2}^{2}}{\Lambda_{-}(x)}\ =\
E_z + 2 \frac{c_{1}^{2}+c_{2}^{2}}{\mu (x)}.}
\end{array}
\]
Thus, by (\ref{e4}) there exist suitable constants
$\bar{\lambda}, \bar{k}>0$ such that:
\[
\|\dot{x}(s)\|_{R}\ \leq\ \bar {\lambda} d(x(s),x(0))+\bar {k}\
\leq\ \bar{\lambda}\int_{0}^{s}\|\dot{x}({r})\|_R\, {\rm d}
{r}+\bar {k} \quad \hbox{for all $s \in [0,T)$.}
\]
In conclusion, we obtain
\[
\log\left(\bar{\lambda}\int_{0}^{s}\|\dot{x}({r})\|_R\, {\rm d}{r}+\bar{k}\right)-\log(\bar{k})\leq
\bar{\lambda}s\leq \bar{\lambda} T\quad \hbox{for all $s \in [0,T)$}
\]
which implies the boundedness of the
$\langle\cdot,\cdot\rangle_{R}$--length of $x(s)$ in $[0,T)$, as
required. 
\end{proof}

\end{document}